\documentclass[12pt]{amsart}

\usepackage{amsmath,amssymb,mathtools}
\usepackage{sagetex}
\usepackage{color}
\usepackage[utf8]{inputenc}

\newtheorem{thm}{Theorem}%[section]
\newtheorem{prop}[thm]{Proposition}%[section]
\newtheorem{cor}[thm]{Corollary}%[section]
\newtheorem*{rem}{Remark}

\newcommand{\Q}{{\mathbb Q}} 
 
\newcommand{\fp}{{\mathbb F}_p} 
\newcommand{\fq}{{\mathbb F}_q}
\newcommand{\fqtwo}{{\mathbb F}_{q^2}}
\newcommand{\fqn}{{\mathbb F}_{q^n}} 
\newcommand{\gal}{{\operatorname{Gal}}}
\newcommand{\frob}{{\operatorname{Frob}}}

\newif\ifsage
\sagefalse

\title[Prime and Möbius correlations for very short intervals]{Prime and Möbius correlations for very short intervals
  in $\fq[x]$.}

\author{P\"ar Kurlberg, Lior Rosenzweig}
\urladdr{www.math.kth.se/\~{ }kurlberg}
\address{Department of Mathematics, KTH Royal Institute of Technology,
SE-100 44 Stockholm, Sweden}
\email{kurlberg@math.kth.se}

\address{Unit of Mathematics, Afeka Tel Aviv College of Engineering, Mivtza Kadesh 38, Tel Aviv, Israel}
\email{liorr@afeka.ac.il}

\thanks{The authors were partially supported by grants from 
the G\"oran Gustafsson Foundation for Research in Natural Sciences and
Medicine,
and the
Swedish Research Council (621-2011-5498, 2016-03701).}

\date{July 1, 2020}
\begin{document}
\begin{abstract}
  We investigate function field analogs of the distribution of primes,
  and prime $k$-tuples, in ``very short intervals'' of the form
  $I(f) := \{ f(x) + a : a \in \fp \}$ for $f(x) \in \fp[x]$ and $p$
  prime, as well as cancellation in sums of function field analogs of
  the Möbius $\mu$ function and its correlations (similar to sums
  appearing in Chowla's conjecture).
  For generic $f$, i.e., for $f$ a Morse polynomial, the
  error terms are roughly of size $O(\sqrt{p})$ (with typical main
  terms of order $p$).
  For non-generic $f$ we prove that independence still holds for
  ``generic'' set of shifts.  We can also exhibit examples
   for which there is no cancellation at all in 
  Möbius/Chowla type sums (in fact, it turns out that (square root)
  cancellation in Möbius sums is {\em equivalent} to (square root)
  cancellation in Chowla type sums),  as well as intervals where the
  heuristic ``primes are 
  independent'' fails badly.
  %
%   Using a Chebotarev density theorem for incomplete intervals (i.e.,
%   for $a$ ranging over intervals of the form $[N, N+M]$, and $M$
%   slightly larger than $\sqrt{p} \log p$), the above
%   results can be shown to also hold for incomplete intervals.

  The results are deduced from a general theorem on correlations of
  arithmetic class functions; these include characteristic functions
  on primes, the Möbius $\mu$ function, and divisor functions (e.g.,
  function field analogs of the Titchmarsh divisor problem can be
  treated.)
  We also prove analogous, but slightly weaker, results in the more
  delicate fixed characteristic setting, i.e., for $f(x) \in \fq[x]$
  and intervals of the form $f(x) +a$ for $a \in \fq$, where $p$ is
  fixed and $q=p^{l}$ grows.
\end{abstract}

\maketitle

% TODO:
% \begin{itemize}
% \item Mention incomplete intervals in intro
%   
% \item Take care of Andrew's remaining comment (``too much 'recall'';
%   not so obvious to him.)  I (PK) will take care of it.
% \end{itemize}

\section{Introduction}

Given a prime $p$, let $\fp$ denote the finite field with $p$ elements,
and let
$$
M_{d} = M_{d}(\fp) := \{ f \in \fp[x] : \text{$f$ is monic and $\deg(f)=d$}  \}
$$
denote the set of monic polynomials of degree $d$.  %For a finite set
%of polynomials $A\subset \fp[x]$, let $\pi(A)$ denote the number
%of irreducible polynomials in $A$.
Gauss gave an exact formula for  the number of prime, or irreducible,
polynomials in 
$M_{d}(\fp)$, namely
$$
% \pi(M_{d}(\fp)) =
| \{ f \in M_{d}(\fp) : \text{ $f$ is prime } \} |
=
\frac{1}{d} \sum_{e|d} \mu(d/e) p^{e}
=
\frac{p^{d} }{d} \cdot (1 + O(p^{-d/2}));
%=
%\frac{|M_{d}(\fp)|}{d} \cdot (1 + O(1/p));
% = \frac{p^{d}}{d} \cdot (1 + O(1/p))
$$
since $|M_{d}(\fp)| = p^{d}$ this can be viewed as a function
field analog of the Prime Number Theorem as $p^d$ tends to infinity, with
$1/d$ playing the role of the "prime density", with square root
cancellation in the error term.
In this paper, we shall be concerned with  ``short interval''
analogs of Gauss' result, 
various generalizations to prime $k$-tuples, square root
cancellation in Möbius $\mu$ sums, as well as
% in their correlations (i.e.,
sums appearing in Chowla's conjecture (these will be described
in detail below.)
Given $f \in \fp[x]$ we define a {\em very short interval} around $f$
as the set
$$
I(f) := % I(f,\fp) :=
\{ f(x) + a : a \in \fp \};
$$
clearly $|I(f)|=p$.  In order to avoid trivialities we will from now
on assume that $\deg(f) \geq 2$.  Further, as we are mainly interested
in the large $p$ limit, we will 
assume that $p>d$ unless otherwise noted
(cf. Section~\ref{sec:large-q} for results when $p$ is fixed but
$q=p^{l}$ grows.)

\subsection{Results for generic intervals}
\label{sec:results-gener-interv}
An element $f \in M_{d}(\fp)$ is said to be a {\em Morse polynomial}
provided that $f$ has $d-1$ distinct critical values, i.e.,  $|\{
f(\xi) : f'(\xi) = 0 \}| = d-1$.  A basic fact
(cf. Section~\ref{sec:gener-morse-polyn}) is that $f$ is Morse for a
{\em generic} choice of coefficients; in particular, given
$f(x) \in M_{d}(\fp)$, the polynomials $f(x) + sx$ will be Morse for
all but $O_{d}(1)$ elements $s \in \fp$. Our first result is that an
analog of the Hardy-Littlewood prime  $k$-tuple conjecture holds
for {\em almost all} very short intervals, namely the ones
``centered'' at Morse polynomials.
For simplicity we state
  the result only for simultaneous prime specialization, but in fact
  any set of $k$ factorization patterns can be treated,
  cf. Section~\ref{sec:corr-class-funct}.

\begin{thm}
  \label{thm:prime-independence}
Assume that $f \in M_{d}(\fp)$ is a Morse polynomial, and  $d \ge 2$. We
then have
\begin{equation}
  \label{eq:primes}
  |\{ g \in I(f) : \text{ $g$ is prime }\}|
  = \frac{p}{d} + O_{d}(\sqrt{p})
\end{equation}
  Moreover, given $k$ distinct shifts $h_{1}, h_{2}, \ldots, h_{k}
  \in \fp$, we have
\begin{multline}
\label{eq:prime-k-tuple}
|\{ g \in I(f) : \text{ $g+h_{1}, g+h_{2}, \ldots, g+h_{k}$ are prime }\}|
 \\ = \frac{p}{d^{k}} + O_{d,k}(\sqrt{p})
\end{multline}
\end{thm}
% % Commented out since the first part goes back to Cohen
% The first assertion can be viewed as a function field analogue of
% Selberg's result \cite{selberg-short-intervals} (conditional on the
% Riemann hypothesis) 
%  on ``almost all'' intervals $[x, x+\Phi(x)]$,
% for $\Phi(x)/(\log x)^{2}$ tending to infinity, containing the expected
% number of primes; the assertion is weaker in the sense that Selberg
% treats much shorter interval, but stronger in that we characterize the
% set of ``good'' intervals.

The latter assertion is a natural function field analogue of the prime
$k$-tuple conjecture for integers in short intervals. However, unlike
the integer case, for $f$ Morse there are {\em no} fluctuations in the
Hardy-Littlewood constants as $h_{1}, \ldots, h_{k}$ varies over
distinct elements. Interestingly, large variations do occur in the
non-Morse case
(cf. Section~\ref{sec:further-degenerate-examples}), and, very
surprisingly, there are non-Morse examples where ``prime independence'' 
breaks down completely for certain rare shifts
(cf. Section~\ref{sec:breakd-indep-prim}.)

We remark that an easy consequence of \eqref{eq:primes} is a
prime number theorem for progressions that is valid for ``very large''
Morse moduli: given $b \in \fp^{\times}$ and a Morse polynomial
$q(x) \in M_{d}(\fp)$,
$$
|\{ a \in \fp : a \cdot q(x) + b \text{ is prime} \}|
=
p/d + O_{d}(\sqrt{p}).
$$

The distribution of primes, and prime $k$-tuples, in ``short
intervals'', i.e., sets of the form
$I(f,1) := \{ f(x) + a_{1} x + a_{0} : a_{0}, a_{1} \in \fq \}$, or
more generally, sets of the form
$I(f,m) := \{ f(x) + \sum_{n=0}^{m} a_{n} x^{n} : a_{0}, \ldots, a_{m}
\in \fq \}$ for $1 \le m <\deg(f)$, has received considerable
attention in the large field limit, i.e., where $q = p^{k} \to \infty$
(in particular allowing for $p$ fixed).
% size of the field considered tends to infinity.
%
That (\ref{eq:primes}) holds for $f$ ``in general'' (i.e., when
$f(x)-t$ has Galois group $S_{d}$ over $\overline{\fq}(t)$) goes back
to Cohen's pioneering work \cite{cohen70}; 
in \cite{cohen-uniform} he showed that it holds for the short interval
$I(f,1)$ provided $f \in M_{d}(\fq)$ and $p>d$.
In \cite{B}
Bary-Soroker removed this size condition for $p$ odd, and 
allowed for more general shifts.
In
\cite{BBR}, the second author, together with Bank and Bary-Soroker,
show that for any prime power $q$, for {\em all} polynomials $f$, and
$m\geq3$
\[
  |\left\{g\in I(f,m):g\mbox{ is prime}\right\}|
  =
  \frac{q^{m+1}}{\deg(f)}+O_{\deg{f}}(q^{m+1/2});
\]
in fact, under minor restrictions on $f$ and $q$ one may take $m=2$ or even
$m=1$ (it is also implicit that (\ref{eq:primes}) holds for 
$f$ Morse.)
An analog of the prime $k$-tuple conjecture for the ``long''
interval $M_d(\fp)$ was shown by Pollack \cite{pollack08} provided
that $(2p,d)=1$.  This co-primality condition was  removed by
Bary-Soroker \cite{B}; Bank and Bary-Soroker then treated the case of
short intervals (i.e., $I(f,m)$, $m \ge 2$) and  $q$ odd in
\cite{BB}.
We also mention that Entin \cite{entin-bateman-horn} has shown prime
$k$-tuple equidistribution for short intervals
in a more general setting, namely for
``Bateman-Horn'' type specializations (e.g., for nonassociate,
separable and irreducible polynomials
$F_{1}(x,t), \ldots, F_{k}(x,t) \in \fq[x,t]$, he obtains the
asymptotics for simultaneous irreducibility 
of the $k$ specialized polynomials $F_{1}(g(t),t), \ldots, F_{k}(g(t),t)$, for
$g \in I(f,m)$); cf. \cite{entin-in-preparation} for recent further
developments.
%
% recently he further
% strengthened these results 
% using a more geometric approach
%, in particular removing the dependence
%on the 
%classification of finite simple groups
%
For a nice survey of recent results on
function field analogs of similar  questions in classical number
theory, including analogs of cancellation in Möbius $\mu$ and Chowla
sums described below, see \cite{R_ICMproceedings}.

A function field analog of the Möbius $\mu$ function on $M_{d}(\fp)$
can be defined as follows: given a  squarefree polynomial
$g \in M_{d}(\fp)$, write $g$ as a product of $l$ distinct
monic irreducibles, i.e., $g = \prod_{i=1}^l g_{i}$, and define
$
\mu(g) := (-1)^{l};
$
if $g$ is not squarefree we set $\mu(g) = 0$.
We then find that there
is square root cancellation in Möbius sums, as well as in the
auto-correlation type sums appearing in Chowla's conjecture
(cf. \cite{C}), for very 
short intervals in the large $p$ limit.
\begin{thm}
\label{thm:moebius-chowla-cancellation}
Assume that $f \in M_{d}(\fp)$ is Morse, and $d \ge 2$.  Then
\begin{equation}
  \label{eq:moebius-sum}
\sum_{g \in I(f)} \mu(g) = O_d(\sqrt{p}).  
\end{equation}
More generally, given distinct elements $h_{1}, h_{2}, \ldots, h_{k}
  \in \fp$, we have
\begin{equation}
  \label{eq:chowla-sum}
\sum_{g \in I(f)} 
\left(  \prod_{i=1}^{k}  \mu(g+h_{i})  \right)
= O_{d,k}(\sqrt{p}).
\end{equation}
\end{thm}
For general $f$ (i.e., non-Morse) we shall see that square root
cancellation in (\ref{eq:moebius-sum}) is {\em equivalent} to square
root cancellation in (\ref{eq:chowla-sum}); moreover either there is
square root cancellation, or there is {\em no cancellation at
  all}. See Section~\ref{sec:lack-canc-mobi} for more details, as well
as examples of intervals on which $\mu$ has constant sign.

In \cite{carmon-rudnick-moebius-chowla}, Carmon and Rudnick showed
that Chowla type sums over $M_d(\fq)$ has square root cancellation as
$q\to \infty$, provided $q$ is odd; in \cite{carmon-char-two}, Carmon
treated even $q$.  In \cite{KR} Keating and Rudnick
proved square root cancellation for Möbius sum over intervals of type
$I(f,m)$ for $m\geq2$; they also gave examples of polynomials $f$ for
which the Möbius sum over $I(f,1)$ has no cancellation at all.
We also note that Entin \cite{entin-bateman-horn} can treat
cancellation in short Chowla type sums in the more general Bateman-Horn
type setting described earlier.

%We note that these results are in fact best possible
%when considering all monic polynomials of fixed degree, as we see in
%this paper, and already observed in \cite{BBR}.\rfixme{Careful here:
%  you (and Lior) also did some best possible stuff, no?}

\subsubsection{Class function correlations}
\label{sec:corr-class-funct}

The above results are easily deduced from a more general result valid
for functions induced from class functions on $S_{d}$, the symmetric
group on $d$ letters.  Briefly, for squarefree $g \in M_{d}(\fp)$ we
associate a conjugacy class $\sigma_{g}$ in $S_{d}$ as follows:
factoring $g$ into prime polynomials, i.e., writing
$g = \prod_{i=1}^{l} P_{i}$, choose $l$ {\em disjoint} cycles
$c_{1}, \ldots, c_{l} \in S_{d}$ such that the length of $c_{i}$
equals $\deg(P_i)$ for $1 \le i \le l$; we then define $\sigma_{g}$ as
the conjugacy class generated by $\prod_{i=1}^{l} c_{i}$.  

Now, given a class function $\phi$ on $S_{d}$ (i.e. $\phi(\sigma)$
only depends on the conjugacy class of $\sigma$), the above
construction allows us to define a function, also denoted $\phi$, on
the set of squarefree elements in $M_{d}(\fp)$. As the number of
non-squarefree polynomials in $I(f)$, for $f \in M_{d}$ is $O_{d}(1)$
(cf. (\ref{eq:square-free-short-interval})) we may then choose
any bounded extension of $\phi$ to $M_{d}(\fp)$. 
In order to simplify statements we will in what follows always 
assume  that the supremum norms of all class functions,
and their extensions, are bounded by some absolute constant.
% , but for
% simplicity we set $\phi$ to zero on non-square free $g$.
%statement of our results we allow this bound to {\em only} depend on
%$d$.
% We call any such function $\phi$ on $M_{d}(\fp)$
% %an {\em
% % arithmetic class function}, or simply
% an {arithmetic class function}, or simply a {\em class function}.

\begin{thm}
  \label{thm:class-function-correlations}
  Assume that $f \in M_{d}(\fp)$ is a Morse polynomial, and $d \ge 2$.
  Further, let $\phi_1, \ldots, \phi_k$ be class functions
  on $S_{d}$,  extended as above to functions
  on $M_{d}(\fp)$ .% so that $\lVert \phi_i \rVert_{\infty} \leq 1$
                   % for $i=1,\ldots, k$.  
Then there exists
  constants $\{c(\phi_i)\}_{i=1}^{k}$, given by
$$c(\phi_i) = 
\frac{1}{|S_{d}|}
\sum_{ \sigma \in S_{d}} \phi_i(\sigma), \quad i=1, \ldots, k.
$$
  such that
  $$
  \sum_{ g \in I(f)} 
  \phi_i(g)
  = p \cdot c(\phi_i) + O_{d}(\sqrt{p})
  $$
  for $i = 1, \ldots, k$.
  Moreover, given distinct elements $h_{1}, h_{2}, \ldots, h_{k}
  \in \fp$, we have
  \begin{equation}
    \label{eq:class-function-independence}
\sum_{ g \in I(f)}
\left(
\prod_{i=1}^{k}  \phi_i(g+h_{i})  
\right)
  = p \cdot \prod_{i=1}^k c(\phi_i) + O_{d,k}(\sqrt{p}).
  \end{equation}
\end{thm}
\noindent
We remark that Theorem~\ref{thm:class-function-correlations} does {\em
  not} hold in the large $q$ limit,
cf. Section~\ref{sec:large-q} for further details, together with a
suitably weakened independence result valid for  the large $q$ limit.

When detecting factorization patterns the constants $c(\phi_{i})$ can
be given a simple combinatorial interpretation.  Namely, given a
desired factorization pattern of $g\in M_{d}(\fp)$, associate an
$S_{d}$-conjugacy class $C$ as described above.  This in turn can be
interpreted as a partition of $d$, i.e., $d = \sum_{j \geq 1} d_{j} j$
(e.g., for the partition $4=2+1+1$, $d_{1} =2$, $d_{2}=1$, and
$d_{j} = 0$ for $j > 2$).  With $\phi = 1_{C}$, where $1_{C}$ denotes
the characteristic function of the conjugacy class $C$, we have
$$
c(\phi) = \frac{|C|}{|S_{d}|}
=
\frac{1}{\prod_{j} (j^{d_{j}} (d_{j}!)) }
$$
(since $|C| = \frac{|S_{d}|}{\prod_{j} j^{d_{j}} (d_{j}!) }$.)  For
example, if 
$C = \{ \sigma \in S_{d} : \sigma \sim (123\ldots d) \}$, we find that
$1_{C} = 1_{\text{Prime}}$ (the characteristic function on the set of
prime polynomials), and $c(1_{\text{Prime}}) = |C|/|S_{d}|= (d-1)!/d!=1/d$.

Other interesting examples of class
functions include
%(note that
%irreducible polynomials correspond to full $d$-cycles),
the Möbius $\mu$ function, as well as the function field analog of
divisor functions $d_{r}$ for integer $r \ge 2$; e.g., $d_{2}(g)$ is
the number of ways to decompose $g$ as a product of two monic
polynomials. In particular, Theorems~\ref{thm:prime-independence} and
\ref{thm:moebius-chowla-cancellation} are immediate consequences of
Theorem~\ref{thm:class-function-correlations}.  In similar fashion we
can treat short interval function field analogs of the ``shifted
divisor problem'', e.g., the sum
$ \sum_{ g \in I(f)} d_{r}(g) d_{r}(g+1) $, as well as the Titchmarsh
divisor problem, e.g., sums of the form
$\sum_{ g \in I(f)} 1_{\text{Prime}}(g) d_{r}(g+1)$.  
These results can be viewed as very short interval versions of recent
results \cite{ABR} by Andrade, Bary-Soroker and Rudnick for the full
interval $M_d(\fq)$.

We remark that Theorem~\ref{thm:class-function-correlations} is, via
the Chebotarev density theorem, Galois
theoretic at heart (cf. Section~\ref{sec:remarks-galo-theory}): to
each polynomial $f(x)+h_{i}+t$ we can associate 
a field extension $L_{h_{i}}/\fp(t)$ with Galois group
$G_{h_{i}} = \gal(L_{h_{i}}/\fp(t)) \simeq S_{d}$, and the independence
implicit in (\ref{eq:class-function-independence}) boils down to
linear independence of the field extensions
$L_{h_{1}}, L_{h_{2}}, \ldots, L_{h_{k}}$. In particular, with $L^{k}$
denoting the compositium of $L_{h_{1}}, \ldots, L_{h_{k}}$, we have
$\gal(L^k/\fp(t)) \simeq (S_{d})^{k}$.

\subsection{Independence results for non-generic intervals}
\label{sec:results-non-generic}

For non-Morse polynomials the situation is more complicated since
$G_{h_{i}}$
%
%(recall the notations introduced in the previous paragraph)
%
might be smaller than $S_{d}$, and
% In particular the correspondence
% between factorization types and conjugacy classes in $S_{d}$ breaks
% down, e.g., the cycles $(123)$ and $(132)$ are conjugate in $S_{3}$
% but not in $A_{3}$.
% %as the latter group is abelian.
%Moreover, even though the isomorphism class of
%$G_{h_{i}}$ does not change with $h_{i}$, 
$\gal(L^k/\fp(t))$  is in general {\em not}
a product of groups.
However, while independence can fail for non-Morse polynomials
(cf. Section~\ref{sec:breakd-indep-prim}), we can still show that
independence holds for ``generic'' choices of distinct shifts
$h_{1}, \ldots, h_{k} \in \fp$ and $p$ large.
%
%%  To describe the result we need some further notation. Given $f \in
%%  M_{d}(\fp)$, let $L/\fp(t)$
%%  denote the splitting field extension generated by $f(x)+t$, and let
%%  $G := \gal(L/\fp(t)$ denote its Galois group. Then
%%  $G_{h_{i}} \simeq G$ for $i=1,\ldots, k$; after chosing $k$ such isomorphism
%%  we may  identify each $G_{h_{i}}$ with $G$.
%
% Given class functions
% $\phi_1, \ldots, \phi_k$ on $G$ (i.e., functions that are constant on
% $G$-conjugacy classes)
\begin{thm}
\label{thm:non-generic-correlations}
  Let $d \ge 2$, and let $\phi_1, \ldots, \phi_k$ be class functions on
  $S_{d}$, extended as before to functions on $M_{d}(\fp)$.
  % $S_{d}$.
  There exists {\bf finite} sets %of constants
  $C(\phi_1,d), \ldots, C(\phi_k,d)$ (with $C(\phi_i,d)$ only
  depending on $\phi_{i},d$) 
  such that the following holds: 
  %For  any $p$, for any $d$-bounded extension
%  of $\phi_1,\dots,\phi_k$ as 
%  above to functions on $M_{d}(\fp)$, and
%
%  For any prime $p$ and  $f \in M_{d}(\fp)$,
 For $f \in M_{d}(\fp)$, 
  %such that
%  $|C(\phi_i,d)| = O_{d}(1)$ for $i=1, \ldots, k$
  $$
  \sum_{ g \in I(f)} 
  \phi_i(g)
  = p \cdot c_i + O_{d}(\sqrt{p}),
  $$
  where $c_i \in C(\phi_i,d)$, for $i = 1, \ldots, k$.
  Moreover, there exists a set
  $B(f) \subset \fp$, of cardinality at most
  $(d-1)^{2}$, with the following property: given distinct elements
  $h_{1}, h_{2}, \ldots, h_{k} \in \fp$ such that
  $h_{i}-h_{j} \not\in B(f)$ for $i\neq j$, we have
$$
\sum_{ g \in I(f)}
\left(
\prod_{i=1}^{k}  \phi_i(g+h_{i})  
\right)
  = p \cdot \prod_{i=1}^k  c_i + O_{d,k}(\sqrt{p}).
  $$
\end{thm}

Note that the number of distinct shifts
$h_{1}, \ldots, h_{k} \in \fp$ such that $h_{i}-h_{j} \in B(f)$ is
$O_{k,d}(p^{k-1})$, hence  independence holds for most choices of
shifts.

Determining the constants $c_{i}$ is
delicate\footnote{E.g., some  factorization patterns
might  not occur at all, cf. Section~\ref{sec:further-degenerate-examples}.}
and requires some knowledge about $G_{h_{1}} = \gal(L_{h_1}/\fp(t))$
(it turns out that the isomorphism class of $G_{h_{i}}$ does not change with
$h_{i}$.)
% Using the 
% fact that the isomorphism class of $G_{h_{i}}$ does not change with
% $h_{i}$, we may --- after making {\em non-canonical} labeling of the
% roots of $f(x)+h_{1}+t, \ldots, f(x)+h_{k}+t$ --- make 
% identifications
% $$
% G_{h_{i}} \simeq G_{h_{1}} \subset S_{d}.
% $$
With $l_{h_{1}} := L_{h_{1}} \cap \overline{\fp}$ denoting the field
of constants in $L_{h_{1}}$, let
$G_{h_{1},\text{geom}}:= \gal(L_{h_{1}}/l_{h_{1}}(t))$ denote the ``geometric
part'' of $G_{h_{1}}$. After making a {\em non-canonical} labeling of
the roots of $f(x)+h_{1}+t$ and $f(x)+h_{i}+t$ (regarded as 
polynomials with coefficients in $\fp(t)$), we obtain an identification
and inclusion $G_{h_{i}} \simeq G_{h_{1}} \subset S_{d}$ and can write
\begin{equation}
  \label{eq:constant-via-coset-sum}
c_{i} =
\frac{1}{|G_{h_{1},\text{geom}}|}
\sum_{\sigma \in \tau \cdot G_{h_{1},\text{geom}}} \phi_i(\sigma)
\end{equation}
%where $G_{h_{1},\text{geom}}$ is a certain normal subgroup of
%$G_{h_{1}}$, and
where $\tau \in G_{h_{1}}$ is any element such that
$\tau|_{l_{h_{1}}}$ 
acts as Frobenious on the finite field extension
$l_{h_{1}}/\fp$, i.e., $\tau(\alpha) = \alpha^{p}$ for $\alpha \in
l_{h_{1}}$.
% (here we have used  $\gal(l_{h_{1}}(t)/\fp(t))
% \simeq \gal(l_{h_{1}}/\fp)$.)
%
For some examples where  class function constants are computed using
Galois theory,
see Sections~\ref{sec:class-funct-const} and
\ref{sec:miss-fact-types}.

The independence can also be explained in terms of Galois
theory. Briefly, after making non-canonical identifications
$G_{h_{i}} \simeq G_{h_{1}}$ for $i=2,3,\ldots,k$, we obtain
inclusions
$$
\gal(L^k/\fp(t)) \subset \prod_{i=1}^{k} G_{h_{i}} \subset (G_{h_{1}})^{k}
$$
and the independence amounts to Frobenius equidistribution inside the
coset $(\tau \cdot G_{h_{1},\text{geom}})^{k}$.  We note that the
methods (cf. the remark after Proposition~\ref{prop:uncorrelated-class-average})
allows us to take $\phi_1,\ldots,\phi_k$ to be class functions 
on $G_{h_{1}}, \ldots, G_{h_{k}}$, rather than on $S_{d}$, and this
sometimes allows for going beyond factorization patterns.
%allows for more refined information about factorizations.
E.g., the cycles $(123)$ and $(132)$ are conjugate in $S_{3}$, but not
in $A_{3}$ (the latter group is abelian); when
$G_{h_{i}} \simeq A_{3}$, after a non-canonical labeling of the roots,
we can distinguish the two cases in terms of the Frobenious action on
the roots.  Another example is given in Section~\ref{sec:going-beyond-fact}.

A more detailed discussion, in particular regarding  the set
$B(f)$ can be found in % Section~\ref{sec:gener-morse-polyn}.
Sections~\ref{sec:remarks-galo-theory} and \ref{sec:class-functions}.

%We next turn to  settings where independence breaks down.

\subsection{Lack of cancellation in Möbius and Chowla sums}
\label{sec:lack-canc-mobi}

An unexpected phenomena is the existence of elements
$f \in M_{d}(\fp)$ for which there is {\em no
  cancellation}
in short interval Möbius sum, i.e., 
$$
\left|
  \sum_{g \in I(f)} \mu(g) 
\right|
=
p + O_{d}(1).
$$
%\footnote{

For example, for $d$ odd and
  $p \equiv 1 \mod d$, take $f(x) = x^{d}$
  (cf. Sections~\ref{sec:lack-cancellation-moebius} and
  \ref{sec:canc-chowla-sums}).  Even more surprising, as 
  noted in \cite{KR}, for $f(x) = x^{2p}$ there is complete lack of
  cancellation for the sum over the {longer} interval $I(f,1)$.
%}
%
%
In fact, either there is square root cancellation in both the Möbius
sum as well as the Chowla sum, or there is essentially no cancellation
whatsoever in either sum
(cf. Theorem~\ref{thm:moebius-chowla-cancellation}.)
\begin{thm}
  \label{thm:moebius-chowla-no-cancellation}
  Let $f \in M_{d}(\fp)$ for $d \ge 2$, and let 
  $h_{1}, \ldots, h_{k} \in \fp$ be distinct   elements.
  Then one of the following occurs: either both
  $$
  \left| \sum_{g \in I(f)} \mu(g) \right| =  p + O_{d}(1), \quad
  \left|
\sum_{g \in I(f)} 
\left(  \prod_{i=1}^{k}  \mu(g+h_{i})  \right)
\right|
= p + O_{k,d}(1)
$$
holds, or both
$$
  \left| \sum_{g \in I(f)} \mu(g) \right| =  O_{d}(\sqrt{p}), \quad
  \left|
\sum_{g \in I(f)} 
\left(  \prod_{i=1}^{k}  \mu(g+h_{i})  \right)
\right|
=   O_{k,d}(\sqrt{p})
$$
holds.  
\end{thm}
We remark that lack of cancellation is
equivalent to the ``geometric part'' of a certain Galois group being
contained in the alternating group $A_{d}$. More details on this, as
well as the proof of Theorem~\ref{thm:moebius-chowla-no-cancellation}
can be found in Section~\ref{sec:mobius-chowla-type}.  Moreover, we
note that Theorem~\ref{thm:moebius-chowla-no-cancellation} is {\em not
true} in the large $q$ limit (i.e., for $p$ fixed),
cf. Section~\ref{sec:large-q}.

\subsection{Further examples of degenerate intervals}
\label{sec:further-degenerate-examples}
We next give some additional examples of short intervals exhibiting
irregular behavior.  For more details regarding these examples,
see Section~\ref{sec:exampl-degen-interv}.

\subsubsection{Prime density fluctuations}
\label{sec:prime-dens-fluct}

Let $f(x) = x^{3}$ and take
$\phi_1=\phi_2 =  1_{\text{Prime}}$. 
Here the constants vary with 
$p$, namely $c(1_{\text{Prime}},p) = 2/3$ for $p \equiv 1 \mod
3$, whereas $c(1_{\text{Prime}},p) = 0$ for $p \equiv 2 \mod
3$.
In fact, there are {\em no primes} in
% the short interval
$I(f)$ if $p \equiv 2 \mod 3$, and in this case the second
part of Theorem~\ref{thm:non-generic-correlations} is trivial.  On
the other hand, it can be shown that $B(f) = \emptyset$ and hence,
for $p \equiv 1 \mod 3$ and $h \not \equiv 0 \mod p$,
\begin{equation}
  \label{eq:kummer-uncorrelated}
\sum_{g \in I(f)} 1_{\text{Prime}}(g) 1_{\text{Prime}}(g+h)
=
(2/3)^{2} \cdot p + O(\sqrt{p}).
\end{equation}

In other words, after taking into account the larger than expected
prime density (for generic degree $3$ polynomials it is $1/3$), the
short interval contains the expected number of twin primes (and
similarly for prime $k$-tuples) --- the heuristic ``primes are
independent'' indeed holds in $I(f)$ as $p \to \infty$, even
though $f(x)=x^{3}$ is {\em not} Morse.

\subsubsection{Lack of cancellation in Möbius sums}
\label{sec:lack-cancellation-moebius}
Again we take $f(x) = x^{3}$ and, as noted before, for
$p \equiv 1 \mod 3$, either $f(x)+a$ splits completely or is irreducible.
In either case, $f(x)+a$ factors into an odd number of irreducibles and
hence $\mu(f(x)+a) = -1$ if $f(x)+a$ is square free, i.e., for all
nonzero $a \in \fp$.
If $p \equiv 2 \mod 3$, $x^{3}+a$ is a permutation for all
$a \in \fp$.  Consequently for all nonzero $a$, $f(x)+a$ has one
linear factor, and one irreducible quadratic factor, and thus
$\mu( x^{3}+a) = 1$ for all nonzero $a \in \fp$.

\subsubsection{Class function constants via Galois theory}
\label{sec:class-funct-const}
To illustrate how averages over cosets of the geometric part of
$G_{0}$ determines the class function constants
(cf. (\ref{eq:constant-via-coset-sum})) we return to the example
$f(x)=x^{3}$.  Then $L_{0} = \fp(t, \zeta_3, \sqrt[3]{-t})$, where
$\zeta_{3}$ denotes a non-trivial third root of unity, and
$l_{0} = L_{0} \cap\overline{\fp} = \fp(\zeta_{3})$.  If
$p \equiv 1 \mod 3$, we have $\zeta_{3} \in \fp$, hence $l_{0} = \fp$,
and $G_{0} = G_{0,\text{geom}} \simeq A_{3}$.  Letting
$\phi_1,\phi_2,\phi_{3}$ denote characteristic functions of the three
$S_{3}$-conjugacy classes $\{()\}$, $\{(123), (132)\}$, and
$\{(12),(13),(23)\}$, the corresponding class function constants given
by Theorem~\ref{thm:non-generic-correlations} and
(\ref{eq:constant-via-coset-sum}) is then given by
$ c_{1} = 1/3, c_{2} = 2/3, c_{3} = 0 $ (the key point is that
Frobenious equidistributes in $A_{3}$.)

If $p \equiv 2 \mod 3$, then
$l_{0} = \fp(\zeta_{3}) = {\mathbb F}_{p^{2}}$, and thus
$G_{0} \simeq S_{3}$ and $G_{0,\text{geom}} \simeq A_{3}$.  Further,
as the action of the Frobenious map $\alpha \to \alpha^{p}$ must act
nontrivially on $l_{0}$, the image of Frobenious equidistributes in
the single conjugacy class given by the non-trival coset of $A_{3}$
(in $S_{3}$), consisting of the three transpositions
$\{(12),(13),(23)\}$.  Hence, for $p \equiv 2 \mod 3$, we have
$ c_{1} = c_{2} = 0, c_{3} = 1 $.

\subsubsection{Class function constants and ``missing'' factorization patterns}
\label{sec:miss-fact-types}
Let $p$ be a large prime and let $f(x) = x^{4}-2x^{2}$.  Then
$\gal(f(x)+t, \fp(t))$ is isomorphic to $D_{4}$, the dihedral group
with $8$ elements.  Regarding $D_{4}$ as a subgroup of $S_{4}$, the
elements of $D_{4}$, in cycle notation, are
$$\{ (1,4)(2,3), (1,3)(2,4), (1,3), (2,4), (1,2)(3,4),
{(1,2,3,4)}, {(1,4,3,2)} \}.$$
Parametrizing the factorization patterns of $f(x)+a$, for $a \in \fp$,
by partitions of $4$, we find that the different factorization
patterns occurs with the following frequencies: $4 = 1+1+1+1$: $1/8$,
$4=2+1+1$: $2/8$, $4=3+1$: $0/8$, $4=2+2$: $3/8$, and finally
$4=4$: $2/8$. In particular, $f(x)+a$ {\em cannot} split into a linear
and a  cubic (irreducible) factor.

Let $\phi_{1}, \ldots, \phi_{5}$ denote class functions (in $S_{4}$)
that equals one on all permutations corresponding to the factorization
pattern given by the $5$ different partitions of $4$ (see above), and
zero otherwise.  The corresponding class function constants in
Theorem~\ref{thm:non-generic-correlations} are then the same as the
corresponding frequencies listed above, and thus $c_{1} = 1/8$,
$c_{2}=2/8$, $c_{3} = 0$, $c_{4}=3/8$, and $c_{5} = 2/8$.

\subsubsection{Going beyond factorization patterns}
\label{sec:going-beyond-fact}
Again take $f(x) = x^{4}-2x^{2}$; as noted above we then have
$\gal(f(x)+t, \fp(t)) \simeq D_{4}$.  The elements of $D_{4}$ that are
products of two disjoint transpositions fall into two $D_{4}$
conjugacy classes, namely $\{ (1,2)(3,4), (1,4)(2,3) \}$ and
$\{ (1,3)(2,4) \}$; these two cases (after labeling the roots) can then
be distinguished if we take class functions on $D_{4}$ rather than on
$S_{4}$.

\subsubsection{Breakdown of independence of primes}
\label{sec:breakd-indep-prim}
For general $f$ the issue of independence for ``bad shifts'' appears
delicate, but we can give an explicit example of a polynomial
$f \in M_{4}(\fp)$ for which the interval $I(f)$ has the expected
prime density, yet prime independence breaks down for a few ``bad''
shifts $h$ --- there can be large fluctuations in the Hardy-Littlewood
constants for $f$
non-Morse.
%(Cf. Section~\ref{sec:breakd-indep-prim-proofs} for full
%details regarding the example below.)
 
Again let $f(x) = x^{4} - 2x^{2}$ and first consider primes
$p \equiv 1 \mod 4$; abusing notation we will let $f=f_{p}$ denote the
reduction of $f$ modulo $p$.
Then
\begin{equation}
  \label{eq:weird-example-prime-sum}
\sum_{g \in I(f)} 1_{\text{Prime}}(g) = \frac{1}{4} \cdot p +  O(\sqrt{p})
\end{equation}
and for $h \in \fp \setminus \{0, \pm 1  \}$ we have
$$
\sum_{g \in I(f)} 1_{\text{Prime}}(g) \cdot 1_{\text{Prime}}(g+h)
= \frac{1}{4^{2}} \cdot p + O(\sqrt{p}),
$$
i.e., prime independence holds.
However, for $h = \pm 1$, we have
$$
\sum_{g \in I(f)} 1_{\text{Prime}}(g) \cdot 1_{\text{Prime}}(g+h)
= \frac{1}{8} \cdot p + O(\sqrt{p})
$$
and independence is clearly violated.

On the other hand, for $p \equiv 3 \mod 4$,
the prime
density is still $1/4$ (e.g., (\ref{eq:weird-example-prime-sum})
holds), but if
$h = \pm 1$, then
$$
\sum_{g \in I(f)} 1_{\text{Prime}}(g) \cdot 1_{\text{Prime}}(g+h)
=  O(\sqrt{p});
$$
in a sense independence is violated in the worst possible way as the
``twin prime constant'' is {\em zero}.

We remark that the $p$-averaged twin prime constant (asymptotically
$p \equiv 1 \mod 4$ holds for half the primes) equals
$ 1/2 \cdot 1/8 + 1/2 \cdot 0 = 1/4^{2}, $ i.e., we arrive at the
expected ``independent'' density --- this is no coincidence, 
cf. Section~\ref{sec:breakd-indep-prim-proofs}.

\subsection{Acknowledgments}
We thank L. Bary-Soroker, A. Granville, and Z. Rudnick for stimulating
and fruitful discussions, as well as their comments on an early draft, and
L. Klurman for pointing out the application to primes in progressions
to very large moduli.

\section{Preliminaries}
\label{sec:preliminary-results}

% \subsection{squarefree polynomials in $I(f)$}
\subsection{Squarefree polynomials in very short intervals}
\label{sec:square-free-polyn}

As we are concerned with class functions on very short intervals we
begin by recording the useful fact that almost all $g \in I(f)$ are
squarefree, for $f \in M_{d}(\fq)$ and $q$ large.  In fact, given
$f \in M_{d}$ and distinct shifts $h_{1}, \ldots, h_{k} \in \fq$,
\begin{equation}
  \label{eq:square-free-short-interval}
|\{ g \in I(f) : \text{$g+h_{1}, \ldots, g+h_{k}$ are squarefree}  \}|
= q + O_{k,d}(1)
\end{equation}
To see this it is enough to verify that $(f+h, f') = 1$ for
all but $O_{d}(1)$ choices of $h \in \fq$, but this is clear as
$f'(\xi)=0$ for at most $d-1$ values of $\xi \in \overline{\fq}$,
so the number of $h$ so that  $f(\xi) + h = 0$ is at most $d-1$.

\subsection{Morse polynomials are generic}
\label{sec:gener-morse-polyn}
As recalled in the introduction, 
%Recall that
 a polynomial of degree $d$ is
called a Morse polynomial if the set of critical values is of
cardinality $d-1$. It turns out that for $f$ a Morse polynomial, the Galois
group of $f(x)-t$ is  maximal (over $\Q(t)$ this goes back to Hilbert
\cite{hilbert-full-galois}.) 
\begin{prop}[Cf. 
  \cite{serre-topics-in-galois-theory-book}, Theorem
  4.4.5]\label{prop:Morse-Sn}  
Assume that $(q,2d)=1$ and that $f\in M_{d}(\fq)$ is a Morse
polynomial.  
%$d$ coprime to $\mbox{char}(\mathbb{F}) > 2$.
Then
$\gal(f(x)-t/\mathbb{F}_{q}(t))\simeq S_d$. 
\end{prop}
We remark that Geyer, in the appendix of \cite{JarRaz00}, also treats
the case $(q,d)=1$ by introducing a more general notion of
Morseness, namely assuming non-vanishing of the second
Hasse-Schmidt derivative of $f$.
Moreover, he also gives a beautiful Galois theoretic proof
that ``generic'' polynomials are Morse.
\begin{prop}\label{prop:Morse-dense}
  Let $f(x)\in M_{d}(\fq)$ with
  $f''(x)\neq0$, and assume that $(q,2d)=1$.  Then, for all but
  $O_{d}(1)$ values of 
  $s\in\overline{\mathbb{F}_{q}}$, the polynomial $f_s(x)=f(x)+sx$ is a
  Morse polynomial.
\end{prop}
Although not stated this way, Proposition \ref{prop:Morse-dense} is in
fact proved in the last page of the proof of Proposition 4.3 in
\cite{JarRaz00}. 

Similar criteria for showing that 
$\gal(f(x) + t x^{m}/\fp(t)) \simeq S_{d}$ for ``generic'' $f$ and
integer $m \in [1,d-1]$ can be
found in \cite[Section~5]{kr-incomplete-chebotarev}.

\subsection{Galois theory and the Chebotarev density
  theorem}
\label{sec:remarks-galo-theory}

For the convenience of the reader, we collect here some results about
Galois groups of function fields over finite fields. Before doing so,
we begin with the following notations, similar to the ones 
used in \cite{GK13,K09}.

For $f \in M_{d}(\fp)$ and $h \in \fp$, define $F_{h}(x,t) \in
\fp[x,t]$ by 
$$F_{h}(x,t) := f(x)+h + t.$$
Set $K_{h} = \fp(t)[x]/(F_{h}(x,t))$, let $L_{h}$ denote its Galois
closure, and let $l_{h} := L_{h} \cap \overline{\fp}$ be the
corresponding field of constants.  As $l_{h}$ is independent of $h$
(cf. \cite[Lemma~5]{K09}), it is convenient to define $l = l_{0}$.

Given $k$ distinct shifts $h_{1}, \ldots, h_{k} \in \fp$, let
$L^{k} := L_{h_{1}} \cdot L_{h_{2}} \ldots \cdot L_{h_{k}}$ be the
compositum of the fields $L_{h_{1}}, \ldots, L_{h_{k}}$,  and let
$G^{k} = \gal(L^{k}/\fp(t))$.  Note that $G^{k}$  is {\em
  not necessarily} a product of groups.  We also define
$G_{h_{i}} := \gal(L_{h_{i}}/\fp(t))$ for $i=1, \ldots, k$; after
labeling the roots of $F_{h}(x,t)$ we obtain a natural inclusion
$G_{h} \xhookrightarrow {} S_{d}$; similarly we
obtain a natural inclusion
$G^{k} \xhookrightarrow {} S_{d}^{k}$.

Let $l^{k} := L^{k} \cap \overline{\fp}$ denote the field of constants
in $L^{k}$, 
%$l$ denote the field of constants in
%$L_{h_{i}}$ for $i=1,\ldots, k$ (note that they have the same constant field).
and let $G^{k}_{\text{geom}} := \gal(L^k/l^k(t))$ denote
the geometric part of $G^{k}$.  Similarly let
$G_{h_{i}, \text{geom}} := \gal(L_{h_{i}}/l_{h_{i}}(t)) =
\gal(L_{h_{i}}/l(t)) $ denote the geometric 
parts of $G_{h_{i}}$, for $i=1,\ldots, k$.  (Here we use  that $l_{h}$ does
not depend on $h$, and that $l=l_{0}$.)

% After labeling the roots of $F_{h_{i}}(x,t)$, for $i=1,\ldots,k$, we
% obtain a natural inclusion $G^{k} \xhookrightarrow {} S_{d}^{k}$. 

The set of critical values of $f$ is given by 
$$
R_{f} := \{ f(\xi) : \xi \in \overline{\fp}, f'(\xi) = 0 \};
$$ 
we then put
$$
B(f) := ((R_{f} - R_{f}) \setminus \{0\}) \cap \fp,
$$
where $R_{f} - R_{f}$ denotes the  set of
differences $\{ r_{1}-r_{2} : r_{1},r_{2} \in R_{f}\}$.

%We next recall some properties of the Artin symbol.
We shall make use of the following properties of the Artin symbol.
%taking values in conjugacy
%classes in a Galois group of a field extension of $\fp(t)$.
%For all but finitely many
%prime ideals $\mathfrak{m}\subset \fp[T]$, there exists a Frobenius
%conjugacy class   $\frob_{\mathfrak{m}}$.
Let
$F(x,t)\in\fq[x,t]$ be a separable irreducible polynomial, and let $L$
denote its splitting field over $\fq(t)$.
%Given a
%finite separable Galois extension $L$ of $\fp(t)$,
%and let ${O}_L$ denote the integral closure of $\fp[t]$ in $L$.
For all
but finitely many $a\in\fq$, the prime ideal
$\mathfrak{p}_{a}=(t-a) \subset \fq[t]$ is unramified in $L$, yielding
a well defined conjugacy class
$(\frac{L/\fq(t)}{\mathfrak{p}_{a}} ) \in \gal(L/\fq(t))$
--- the Artin symbol.
Moreover, for these choices of $a$ the splitting type, or the cycle
pattern, of the polynomial $F(x,a)$ (i.e., when specializing
$t \to a \in \fq$) is the same as the cycle pattern of
$(\frac{L/\fq(t)}{\mathfrak{p}_{a}} )$, interpreted as a permutation
on the roots of $F(x,t)$.
Further, given a conjugacy class $\mathcal{C} \subset \gal(L/\fq(T))$,
the density of prime ideals for which Artin symbol lies in
$\mathcal{C}$ is given by the Chebotarev density Theorem.
%We state here
%the quantitative version that we use.
\begin{prop}[\cite{FJ}, Proposition~6.4.8.]
  Let $K$ be a function field over $\fq$, let $d=[K:\fq(t)]$, let
  $L/K$ be a finite Galois extension, and let $\mathcal{C}$ be a
  conjugacy class in $\gal(L/K)$.  With $\fqn$ denoting the
  algebraic closure of $\fq$ in $L$, let $m = [L:K\fqn]$.
Let $b$ be a positive integer with $\text{res}_{\fqn}   \tau =
\text{res}_{\fqn}   \frob_q^b$ for each    $\tau \in
\mathcal{C}$.  Let $k$ be a positive integer. If $k \not \equiv b \mod
n$, then $C_{k}(L/K, \mathcal{C})$ is empty.  If $k  \equiv b \mod n$,
then
$$
\left|
|C_k(L/K, \mathcal{C})| - \frac{|\mathcal{C}|}{km} q^k
\right|
<
\frac{2|\mathcal{C}|}{km} ((m+g_{L}) q^{k/2} + m(2g_{K}+1)q^{k/4}+g_{L} +dm  ) .
$$
\end{prop}
Here $\frob_q \in \gal(\overline{\fq}/\fq)$ is the Frobenius map given
by $\frob_q(\alpha) = \alpha^{q}$, $g_{L},g_{K}$ are the genera of the
fields $L,K$, and
$$
C_k(L/K,
\mathcal{C}) = \{ \mathfrak{p} \subset O_{K} : \deg(\mathfrak{p}) = k,
\text{ $\mathfrak{p}$ unramified},
\left(\frac{L/K}{\mathfrak{p}} \right) = \mathcal{C} 
  \}
$$
where $O_{K} \subset K$ is the integral closure of $\fq[t]$ in $K$. In
our applications we will always take $K=\fq(t)$ and in this case
$O_{K} = \fq[t]$.

%where the conjugacy class $(\frac{L/K}{\mathfrak{p}} )$ is the Artin symbol.

For a squarefree polynomial $f \in M_{d}(\fq)$, define a conjugacy
class $\sigma = \sigma_{f} \subset S_{d}$ by the Frobenius action
$\alpha \mapsto \alpha^q$ on the roots of $f$.
%This defines a
%conjugacy class in $S_d$.
%Note that $f$ is irreducible if and only if $\sigma_{f}$ is
%cyclic of length $d$. 
Further, if we let $f_{a}(x) := f(x)+a$, the conjugacy classes
$\sigma_{f_{a}}$ (as $a \in \fp$ ranges over elements such that
$f_{a}$ is squarefree) is the same as the Artin symbols
$(\frac{L/K}{\mathfrak{p}_{a}} )$ as $a \in \fp$ ranges over elements
for which the prime ideal $\mathfrak{p}_{a} := (t-a)$ is unramified,
if we take $K = \fp(t)$ and $L = L_{0}$ with notation as above (also
note that $m = |G_{0,\text{geom}}|$ in this case.)

We next collect some crucial facts about the Galois extensions
introduced above.    Given a finite extension $\mathbb{E}/\fp$ it will be
convenient to let $\frob_{\mathbb{E}}$ denote the map $\alpha \to
\alpha^{|\mathbb{E}|}$. 
\begin{prop}[\cite{K09}, Section 2]\label{prop:linear-disjoint-fields}
	Let $f\in M_d(\fp)$.  Then
	\begin{enumerate}
		\item For any $h\in\mathbb{F}_{p}$, $l_{h}=l_{0}$ and
                  $G_h\simeq G_0$.  
		\item If ${\bf{h}}=(h_1,\dots,h_{k})$ is such that
		$h_i-h_j\notin B(f)$, then the field extensions
		$L_{h_1}/l(T),\dots,L_{h_{k}}/l(T)$ are linearly disjoint,
		where $l=l_{0}$ is the field of constants of $L^k$. In
                particular,
                $$
G^{k}_{\text{geom}} = \prod_{i=1}^{k} G_{h_{i},\text{geom}} \simeq (G_{0,\text{geom}})^{k}
                $$
              \item For ${\bf{h}}=(h_1,\dots,h_{k})$ such that
                $h_i-h_j\notin B(f)$, let
                $\mathcal{C}\subset G^k$ be a conjugacy class of
                the form
                $\mathcal{C}=\mathcal{C}_1\times\dots\mathcal{C}_k$,
                where each $\mathcal{C}_i$ is the corresponding conjugacy class
                in $G_{h_i}$ (i.e., where
                $G^k <  G_{h_1}\times\cdots\times G_{h_k}$ and
                $\mathcal{C}_i=\pi_i(\mathcal{C})$ is the image of
                $\mathcal{C}$ under the $i$-th projection.)
%                Then for any ${\bf{h}}=(h_1,\dots,h_{k})$ such that
%                $h_i-h_j\notin B(f)$ the set
                Then
		\[
		\left\{\gamma\in
                  G^k:\gamma|_{l^k}=\frob_{l^k},\gamma|_{L_{h_i}}\in\mathcal{C}_{h_i}\;\forall 
                  i=1,\dots,k\right\} 
		\] 
		is in $1-1$ correspondence with
		\[
		\prod_{i=1}^k\left\{\gamma\in G_{0}:\gamma|_{l}=\frob_l,\gamma\in\mathcal{C}_i\right\}
		\]
		which, if we let $\delta \in G_{0}$ denote any element
                such that $\delta|_{l} = \frob_l$, is in $1-1$
                correspondence 
                with
		\[
		\prod_{i=1}^k\left((\delta \cdot G_{0,geom})\cap \mathcal{C}_i\right),
		\]
	\end{enumerate}
\end{prop}
\begin{proof}
  The proof of the proposition is the content of Lemma~5,
  Proposition~8, and (the proof of) Lemma~10 in \cite{K09}. We note
  that in Proposition 8 and Lemma 10, the first author shows that if
  $R_f+h_1,\dots, R_f+h_k$ are pairwise disjoint (more precisely, he
  considers $h_1=0$, and $F_h(x,t)=f(x)-(h+t)$, and therefore the sets
  are of the form $R_f-h_i$), then linear disjointedness holds, and from
  that also Lemma 10. We note that the sets are indeed pairwise
  disjoint if $h_i-h_j\notin B(f)$ for $i\neq j$.
      \end{proof}

To prove independence when  disjoint ramification does not hold, we
need the following key result (cf. \cite{GK13}, Proposition~17 and
Lemma~16.) 
\begin{prop}
  \label{prop:linear-disjoint-morse}
  If $f \in M_{d}(\fp)$ is a Morse polynomial and $h_{1},\ldots,h_{k}
  \in \fp$ are distinct then
  $G^{k} = S_{d}^{k}$ provided that  %$p \gg_{d,k} 1$.
  $p > 4^{k+d-1} +1$.
\end{prop}

\subsection{Class functions}\label{sec:class-functions}

As mentioned in the introduction, any class function on $S_d$ can be
viewed as an arithmetic class function on the set of squarefree
polynomials in $M_d(\fp)$; we then consider any 
bounded (by some absolute constant) extension
% of such arithmetic class functions 
to the set of
all polynomials in $M_d(\fp)$.
% As $f \in M_d(\fp)$ is prime if and only if the Frobenius action on
% the roots of $f$ is a $d$-cycle, 
% $1_{\text{Prime}}$, the characteristic function on the set of prime
% polynomials, is indeed a class function.
% TODO:
% \begin{itemize}
% \item Sort out how the constants arise as class functions composed 
%   with equidistribution on a coset of $G_{\text{geom}}$.
% 
%   
% \item blabla, blabbla:
% More precisely, given a squarefree polynomial $f \in M_{d}(\fp)$, define 
% $\sigma = \sigma_{f} \in S_{d}$ by the Frobenius action on the roots of $f$
% (strictly speaking, we obtain a conjugacy class in
% $S_{d}$).
%  
%\end{itemize}

%collecting some results:

\begin{prop}\label{prop:all-possible-densities}
  Let $f \in M_{d}(\fp)$, $d \ge 2$ and let $\phi$ be a class function
  on $S_{d}$.  Choose $\gamma \in G_{0}$ such that $\gamma|_{l}$ acts
  via $\alpha \to \alpha^{p}$, and let $G_{0,\text{geom}}$ denote the
  geometric part of $G_{0}$.  Then
  $$
  \sum_{g \in I(f)} \phi(g) =
  c(\phi) \cdot p + O_{d}(\sqrt{p})
  $$
where
  $$
  c(\phi) =
  \frac{1}{|G_{0,\text{geom}}|}
  \sum_{ \sigma \in \gamma \cdot G_{0,\text{geom}} }
  \phi(\sigma)
  $$
\end{prop}
\begin{proof}
The contribution from non-squarefree $g\in I(f)$ is $O_{d}(1)$ (cf.
(\ref{eq:square-free-short-interval}).)  The result now follows from
the Chebotarev density theorem.
\end{proof} {\em Remark:} If $f$ is Morse and $h \in \fp$, then
$G_{h,\text{geom}} = S_{d}$.  In the non-Morse case, the set of
possible constants $C(\phi,d)$ can be shown to only depend on $\phi$
and $d$ by noting that $C(\phi,d)$ is a subset of
\begin{equation}
  \label{eq:possible-constants}
\left\{ \frac{1}{|H|}\sum_{\sigma \in \gamma H} \phi(\sigma) : \gamma \in
S_{d}, \text{ $H< S_{d}$ acts transitively on \{1, \ldots, d\}}
\right\}
\end{equation}

As an immediate consequence of the Chebotarev density theorem we can
give a more precise description of what the constants might be for
several shifts, when independence is allowed to break down.
\begin{prop}
  \label{prop:uncorrelated-class-average}
  Let  $\phi_{1}, \ldots, \phi_k$ be class functions on $S_{d}$ and let  $h_{1},
  \ldots, h_{k} \in \fp$ be distinct shifts.
  Choosing $\gamma \in G^{k}$ such that $\gamma|_{l^{k}}$ acts
  via $\alpha \to \alpha^{p}$, we have
  $$
  \sum_{g \in I(f)}
  \left(  \prod_{i=1}^k \phi_i(g + {h_{i}}) \right)
  =
  c \cdot p + O_{d,k}(\sqrt{p})
  $$
  where
  $$
  c =
\frac{1}{|G^{k}_{\text{geom}}|} \cdot 
  \sum_{\sigma \in \gamma \cdot G^{k}_{\text{geom}}}
  \left(        \prod_{i=1}^k \phi_i(\sigma_{i})  \right),
  $$
and $(\sigma_1, \ldots, \sigma_k) \in S_{d}^{k}$ denotes the image
of $\sigma \in G^{k}$ under the natural inclusion $G^{k}
\xhookrightarrow {} S_{d}^{k}$.
\end{prop}
\begin{rem}
  In order to go beyond factorization patterns (to distinguish
  conjugacy classes having the same factorization pattern), note that
  the proof gives a slightly more general version of
  Proposition~\ref{prop:uncorrelated-class-average}, where
  $\phi_{1}, \ldots, \phi_k$ are class functions on
  $G_{h_{1}}, \ldots, G_{h_{k}}$, and using the inclusion $G^{k}
  \subset G_{h_{1}} \times G_{h_{2}} \times \cdots \times G_{h_{k}}$
  to map $\sigma \in G^{k}$ to $(\sigma_{1}, \ldots, \sigma_k)
  \in G_{h_{1}} \times G_{h_{2}} \times \cdots \times G_{h_{k}}$.
\end{rem}

\section{Proofs of Theorems~\ref{thm:class-function-correlations} and \ref{thm:non-generic-correlations}}
\label{sec:proof-theorem-corr}
We begin with proving Theorem
\ref{thm:class-function-correlations}. The first part of the Theorem
is an immediate corollary of Propositions
\ref{prop:Morse-Sn}  and
\ref{prop:all-possible-densities}. Indeed, for $f\in M_d(\fp)$ Morse,
$G_{0}=G_{0,\text{geom}}=S_{d}$, and the
sums in Theorem \ref{thm:class-function-correlations} and in
Proposition \ref{prop:all-possible-densities} are the same.
As for the second part, for $p$ sufficiently large,
Proposition~\ref{prop:linear-disjoint-morse} gives that $G^{k} = 
G^{k}_{\text{geom}}= 
S_{d}^{k}$. % , and the result  follows from Proposition
%\ref{prop:linear-disjoint-fields}, and Proposition
%\ref{prop:uncorrelated-class-average}. Namely,
By Proposition
\ref{prop:uncorrelated-class-average},
$$
\sum_{ g \in I(f)}
\left(
\prod_{i=1}^{k}  \phi_i(g+h_{i})  
\right)= c \cdot p + O_{d,k}(\sqrt{p})
$$
where
$$
c =
% \frac{1}{|G^{k}_{\text{geom}}|} \cdot 
% \sum_{\sigma \in \gamma \cdot G^{k}_{\text{geom}}}
% \left(        \prod_{i=1}^k \phi_i(\sigma_{i})  \right),
 \frac{1}{|S_{d}^{k}|} \cdot 
 \sum_{\sigma \in  S_{d}^{k}}
 \left(        \prod_{i=1}^k \phi_i(\sigma_{i})  \right)
 =
 \prod_{i=1}^{k} c(\phi_i),
 $$
and $c(\phi_i) =\frac{1}{|S_{d}|} 
 \sum_{\sigma \in  S_{d}} \phi_i(\sigma)$ for $i=1,\ldots k$.

The proof of Theorem~\ref{thm:non-generic-correlations} is similar. 
The first part follows from
Proposition~\ref{prop:all-possible-densities}; 
letting
$c_i =
\frac{1}{|G_{0,\text{geom}}|} \sum_{ \sigma \in \tilde{\gamma} \cdot
  G_{0,\text{geom}} 
}\phi_i(\sigma)
$
it is clear that $C(\phi_i,d)$, the set
of possible values of $c_i$, is clearly a subset of the finite set
given in (\ref{eq:possible-constants}).
As for the second part, we note that, by part (2) of Proposition
\ref{prop:linear-disjoint-fields}, $G_{\text{geom}}^{k} = (G_{0,\text{geom}})^{k}$
and it
follows that the constant in front of $p$ is 
\begin{multline*}
c =
\frac{1}{|G^{k}_{\text{geom}}|} \cdot 
\sum_{\sigma \in \gamma \cdot G^{k}_{\text{geom}}}
\left(        \prod_{i=1}^k \phi_i(\sigma_{i})  \right)
=\\=
\frac{1}{|G_{0,\text{geom}}|^k}\prod_{i=1}^{k}
\left(\sum_{ \sigma \in \tilde{\gamma} \cdot G_{0,\text{geom}}
  }\phi_i(\sigma)\right)
=\prod_{i=1}^{k}c_i,
\end{multline*}
where $\gamma \in G^{k}$ is some element such that $\gamma|_{l^{k}} =
\frob_{l^{k}}$, and $\tilde{\gamma}$ denotes the image of $\gamma$ under the
projection from $G^{k}$ to $G_{0}$.

\section{Proof of Theorem~\ref{thm:prime-independence}}
\label{sec:proof}

Recall that $f\in M_{d}(\fp)$ denotes a Morse polynomial.
We begin by showing that the characteristic function on 
prime 
polynomials is a class function.  With
	\[
	1_{d\text{-cycle}}(\sigma)=\begin{cases}
	1 & \sigma=(i_1\cdots i_d) \text{ is a $d$-cycle}\\
	0 & \text{otherwise}
	\end{cases}
      \]
      the function 
	\[
	\phi(f)=\begin{cases}
	{1}_{d\text{-cycles}}(\sigma_f) & f \text{ is squarefree}\\
	0 & \text{otherwise}
	\end{cases}
      \]
on $M_d(\fp)$       
	is a class function, which equals
        ${1}_{\text{Prime}}$ since a polynomial is irreducible if and
        only if $\sigma_f$ is a $d$-cycle. 
%
% By the remark in Section \ref{sec:class-functions}, the function
% $1_{\text{Prime}}$ is a class function.
%
% As $1_{\text{Prime}}$ is a class function, we find that
        Hence
        $c(1_{\text{Prime}})$ is the density of $d$-cycles in $S_{d}$,
        namely $\frac{1}{d}$, and Theorem \ref{thm:prime-independence} now
        follows from Theorem \ref{thm:class-function-correlations}.

        In a similar way, the ``Titchmarsh divisor problem'', and the ``shifted divisor problem''   for very short intervals $I(f)$
        may be treated.  The former consider sums of the form
\[
\sum_{g\in I(f)}{1}_{\text{Prime}}d_r(f+h)
\]
where $d_r(f)$ is the number of
ways to decompose $f$ as a product of $r$ monic polynomials, 
and the latter concerns sums of the form
\[
\sum_{g\in I(f)}d_{r_1}(f+h_1)\cdots d_{r_k}(f+h_k)
\]
where $r_1,\dots,r_k$ are positive integers, and $h_1,\dots h_k\in\fp$
are distinct. Once again, $d_r(f)$ is a class function, since for 
$f$ squarefree, $d_r(f)=d_r(\sigma_{f})$, where
$d_r(\sigma)$ is the number of ways to decompose the permutation
$\sigma$ as a product of $r$ disjoint cycles (here we allow for empty
cycles.) We can therefore apply 
Theorem \ref{thm:class-function-correlations} for these sums and get
that for distinct $h_1,\dots,h_k \in\fp$,
\begin{equation}
\sum_{g\in I(f)}d_{r_1}(f+h_1)\cdots
d_{r_k}(f+h_k)=\prod_{i=1}^k\binom{n+r_i-1}{r_i-1}p+O(p^{1/2}) 
\end{equation}
(the constants are derived in \cite[Lemma~2.2]{ABR}), and
for all nonzero $h\in\fp$
\begin{equation}
\sum_{g\in
  I(f)}{1}_{\text{Prime}} \cdot d_r(f+h)
=\frac{1}{n}\binom{n+r-1}{r-1}p+O(p^{1/2}).
\end{equation}
\section{Möbius and Chowla type sums in very short intervals}
\label{sec:mobius-chowla-type}
In this section we give proofs of
Theorems~\ref{thm:moebius-chowla-cancellation} and
\ref{thm:moebius-chowla-no-cancellation}.
We begin with a discussion of the Möbius $\mu$ function for function
fields.

Given a polynomial $f \in M_{d}(\fp)$, let $\omega(f)$ denote the
number of distinct irreducible divisors of $f$.
It is natural to define the function field Möbius $\mu$-function by
$\mu(f) := (-1)^{\omega(f)}$ for $f$ squarefree; otherwise we set
$\mu(f)=0$.  For $p$ large and $f \in M_{d}(\fp)$, essentially all
$g \in I(f)$ are squarefree
(cf. (\ref{eq:square-free-short-interval})), and hence $\mu$ is a
class function in the sense previously discussed.

%; this can
%be used to give a combinatorial interpretation of $\mu$.

Given
$\sigma \in S_{d}$, let $c(\sigma)$ denote the number of cycles in the
cycle representation of $\sigma$, {\em including all $1$-cycles}, e.g., for
$(12) \in S_{4}$ we write $(12) = (12)(3)(4)$ and find that
$c(\sigma) = 3$.  Thus $\mu(f)$, for $f$ squarefree, is
given by $(-1)^{c(\sigma_{f})}$.  It is convenient to abuse notation
and define $\mu(\sigma) := (-1)^{c(\sigma)}$ for $\sigma \in S_{d}$.
It turns out that $(-1)^{c(\sigma)}$ is closely related to
$\operatorname{sgn}(\sigma)$, the  sign of $\sigma$ regarded as a
permutation:
% for $\sigma \in S_{d}$, we have
%(the verification is left to the reader)
$$
\mu(\sigma) = 
(-1)^{c(\sigma)} = (-1)^{d} \cdot \operatorname{sgn}(\sigma)
$$ 
To see this consider the disjoint cycle
decomposition $\sigma = \prod_{i=1}^{L} c_{i}$ (including one-cycles).
We then have
$\mu(\sigma) = (-1)^{L}$.
Now, with $L_{1}$ denoting the number of even length cycles, and
$L_{2}$ the number of odd length cycles, we trivially have
$L = L_{1} + L_{2}$ and moreover that
$\operatorname{sgn}(\sigma) = (-1)^{L_{1}}$.
As the sum of the length of all cycles $c_{1}, \ldots, c_{L}$ equals
$d$ (here it is crucial to include one-cycles), we find that $L_{2}$
and $d$ has the same parity.
Hence
$$
\operatorname{sgn}(\sigma) = (-1)^{L_{1}} =
(-1)^{L_{1}+L_{2}+d} =
(-1)^{L+d} =
\mu(\sigma) (-1)^{d} 
$$
and we find that $\mu(\sigma) = \pm \operatorname{sgn(\sigma)}$, where
the sign is given by the parity of $d$.

%To simplify the  notation, let
%$L=L_{0}, G=G_{0}$, and  $G_{\text{geom}} = G_{0,\text{geom}}$.
%
% Let $G_{0}$ denote Galois group of $F_{0}(x,t)$, over $\fp(T)$, and
% let $G_{0,\text{geom}}$ denote the geometric part of
% $G$.
%
Now, if $f$ is Morse, we have $G_{0}= G_{\text{geom}} = S_{d}$, hence
(cf. Theorem~\ref{thm:non-generic-correlations}
%Frobenius equidistributes in $S_{d}$, and as
$$
c(\mu) = \frac{1}{|G|}\sum_{\sigma \in G_{0}} \mu(\sigma) =
\frac{(-1)^{d}}{|S_{d}|}\sum_{\sigma \in S_{d}}
\operatorname{sgn}(\sigma) =  0.
$$
Thus Theorem~\ref{thm:moebius-chowla-cancellation} immediately 
follows from Theorem~\ref{thm:class-function-correlations}.

\subsection{The non-Morse case}
\label{sec:non-morse-case}

We begin by characterizing short intervals on which there is no
cancellation in the sum $\sum_{g \in I(f)} \mu(g)$.
%Recall the
%simplified notation introduced above (i.e., $L=L_{0}, G=G_{0}$ etc)
Fix $\gamma \in G_{0}$ such that $\gamma|_{l}$ acts as
$\alpha \to \alpha^{p}$.

{\em First case:} We begin by considering the case
$G_{\text{0,geom}} \subset A_{d}$ (with $A_{d} \subset S_{d}$ denoting
the alternating group.)  
Since $\mu$ is a class
function, (\ref{eq:constant-via-coset-sum}) gives that 
$$
c(\mu) = \frac{1}{|G_{\text{0,geom}}|}
\sum_{\sigma \in \gamma \cdot G_{\text{0,geom}}} \mu(\sigma)
$$
%
%Recalling that $t \in Frob_{t}$, as $t$
%ranges over elements in $\fp$, equidistributes in the coset
%$\gamma G_{\text{geom}}$,
and since $\operatorname{sgn}$ is trivial on
$G_{\text{0,geom}} \subset A_{d}$, we find that $\mu(g)$ has {\em constant sign} for
$g \in I(f)$, with the possible exception of $O(1)$ non-squarefree
$g$.  Hence $|\sum_{g \in I(f) } \mu(g)| = p + O_{d}(\sqrt{p})$.

{\em Second case:} If $G_{\text{0,geom}}$ is not contained in $A_{d}$,
there exist at least one odd permutation $\tau \in G_{\text{0,geom}}$; in
particular, $\sum_{\sigma \in G_{\text{0,geom}}}
\operatorname{sgn}(\sigma) = 0$.  Thus, 
$$
c(\mu) = \frac{1}{|G_{\text{0,geom}}|}
\sum_{\sigma \in \gamma G_{\text{0,geom}}} \mu(\sigma) =
 \frac{(-1)^{d} \operatorname{sgn}(\gamma)}{|G_{\text{0,geom}}|}
\sum_{\sigma \in G_{\text{0,geom}}} \operatorname{sgn}(\sigma) = 0,
$$
and hence Theorem~\ref{thm:class-function-correlations} gives that
$$
\sum_{g \in I(f) } \mu(g) = O(\sqrt{p})
$$
In summary, there is (square root) cancellation in $\sum_{g \in I(f) }
\mu(g)$ if and only if there is sign cancellation in $\sum_{\sigma \in
  G_{\text{0,geom}}} \operatorname{sgn}(\sigma)$.

\subsubsection{Cancellation in Chowla sums}
\label{sec:canc-chowla-sums}
We first note that if $h_{1}, \ldots, h_{k} \in \fp$ are distinct
elements such that $h_{i}-h_{j} \not \in B(f)$ (``the uncorrelated
case''), Theorem~\ref{thm:moebius-chowla-no-cancellation} follows
immediately from Theorem~\ref{thm:non-generic-correlations}.

If $h_{i}-h_{j} \in B(f)$ (``the correlated case'') we argue as
follows.
%
% Let $L^{k}$ denote the compositum of the extensions
% $L_{h_{1}}, \ldots, L_{h_{k}}$, let $G^{k} = \gal(L^{k}/\fp(t))$, and let
% $G_{h_{i}} = \gal(L_{h_{i}}/\fp(t))$ for $i=1, \ldots, k$.  Further,
% let $l^{k}$ denote the field of constants in $L^{k}$, and
% $l$ denote the field of constants in
% $L_{h_{i}}$ for $i=1,\ldots, k$ (note that they have the same constant field).
% %
% Further, let $G^{k}_{\text{geom}} := \gal(L^k/l^k(t))$ denote
% the geometric part of $G^{k}$, and similarly let
% $G_{h_{i}, \text{geom}} := \gal(L_{h_{i}}/l(t)) $ denote the geometric
% parts of $G_{h_{i}}$, for $i=1,\ldots, k$.
%
As we have seen, no cancellation in the Möbius sum
$\sum_{g \in I(f) } \mu(g+h_{1})$ (which in turn happens if and only
if there is no cancellation in the unshifted sum
$\sum_{g \in I(f) } \mu(g)$, or in any other shifted sum
$\sum_{g \in I(f) } \mu(g+h_{i})$, $i=2,\ldots,k$) is equivalent to
$G_{h_{i},\text{geom}} \subset A_{d}$ for all $i$.  In particular,
$G^{k}_{\text{geom}} \subset \prod_{i=1}^k G_{h_{i},\text{geom}}
\subset A_{d}^{k}$.
Since
$
\sum_{g \in I(f)} \prod_{i=1}^{k} \mu(g + h_{i}) = c \cdot p +
O_{d,k}(\sqrt{p}) 
$
where 
$$
  c =
\frac{1}{|G^{k}_{\text{geom}}|} \cdot 
  \sum_{(\sigma_{1}, \ldots,\sigma_k) \in \gamma \cdot G^{k}_{\text{geom}}}
  \left(        \prod_{i=1}^k \mu(\sigma_{i})  \right),
$$
(cf. Proposition~\ref{prop:uncorrelated-class-average} and use the
natural embedding $G^{k} \xhookrightarrow {} S_{d}^{k})$
%As $Frob_{t}$, for $t$ ranging over elements in
%$\fp$, takes values in a coset of $G^{k}_{\text{geom}}$
we find that
%$\prod_{i=1}^k \mu(\sigma_{i})$ has constant sign 
%and hence
there is no cancellation in the Chowla
sum.

%and similarly 
%$\prod_{i=1}^{k} \mu(g + h_{i})$ has constant sign for $g \in I(f)$
%(apart from $O_{k,d}(1)$ occurrences of $g$ such that $ g + h_{i}$ is
%not squarefree),
%and hence there is no cancellation in the Chowla
%sum.

On the other hand, if there is cancellation in the short Möbius sum,
there must be some odd permutation in $G_{h_{1},\text{geom}}$, and
this in fact implies that the same holds for $G^{k}_{\text{geom}}$,
provided $p$ is sufficiently large (in terms of $k$.)  To see this,
define $R_{\text{odd}} \subset R_{f}$ as the set of critical values of
$f$ giving rise to odd permutations in $G_{h_{1},\text{geom}}$. Then,
as $G_{h_{1},\text{geom}}$ is generated by the inertia subgroups of
points outside $\infty$, and their conjugates (cf.
\cite[Proposition~4.4.6]{serre-topics-in-galois-theory-book}),
$R_{\text{odd}}$ is nonempty.  By \cite[Lemma~16]{GK13} (in
particular, take $H= \{-h_{1}, -h_{2}, \ldots, -h_{k}\}$,
$S=R_{\text{odd}}$ and note that we may assume that
$p > 4^{k+d} \geq 4^{k+|R_{\text{odd}}|}$ since the implied constants
in the error terms are allowed to depend on $d$ and $k$) there must be
some element in the multi-set generated by
$R_{\text{odd}}+h_{1}, \ldots, R_{\text{odd}}+h_{k}$ that has {\em
  odd} parity, and hence $G^{k}_{\text{geom}}$ contains an odd element
given by a product of an odd number of odd permutations.  Thus the
elements of $G^{k}_{\text{geom}}$ do not have constant sign, and hence
there is square root cancellation in the Chowla sum also in this case.

To see that $G_{\text{0,geom}} \subset A_{d}$ does indeed occur (for $p$
large and $d$ fixed; for interesting examples when  $p|d$, see
\cite{carmon-rudnick-moebius-chowla}), we can take $f(x)=x^{l}$ and
$p \equiv 1 \mod l$ for some odd prime $l$; then $G=G_{\text{geom}} $
is cyclic of order $l$, and all nontrivial elements are given by
{\em even} $l$-cycles.

%\section{Further examples of degenerate intervals}
\section{Examples of degenerate intervals --- further details}
\label{sec:exampl-degen-interv}

\subsection{Prime density fluctuations}
%\subsection{Prime density fluctuations and no cancellation in Möbius
%  sums}
\label{sec:prime-dens-fluct-proofs}

We take  $f(x) = x^{3}$, $\phi_1=\phi_2 =  1_{\text{Prime}}$.  As
$I(f) = \{ x^{3} - t, t \in \fp \}$ it is enough to consider splitting
patterns of $x^{3}-t=0$.
For primes $p \equiv 1 \mod 3$, $x^3-t$ has
either zero or three roots in $\fp$; the latter happens if and only if
$t$ is a cube of some element in $\fp$, and there are $1+ (p-1)/3$
such elements.  Hence $c(1_{\text{Prime}},p) = 2/3$ for $p \equiv 1
\mod 3$.
%
% Easy: only EVEN three cycles remain if no roots.
%
On the other hand, for $p \equiv 2 \mod 3$, the map $x \to x^{3}$ is a
permutation of the elements in $\fp$, hence $x^{3}-t=0$ has one root
in $\fp$ no matter what $t$ is.  In particular, $c(1_{\text{Prime}},p)
= 0$ for $p \equiv 2 \mod 3$.

Since $x^{3}$ only has one critical value, $|R_{f}|=1$ and hence
$(h_{1}+R_{f}) \cap (h_{2} + R_{f}) = \emptyset$ unless $h_{1}=h_{2}$; in
particular $B(f) = (R_{f} - R_{f}) \setminus \{0\} = \emptyset$, and
(\ref{eq:kummer-uncorrelated}) follows from
Theorem~\ref{thm:non-generic-correlations}.

\subsection{Breakdown of independence of primes}
\label{sec:breakd-indep-prim-proofs}

Take $f(x) = x^{4} - 2x^{2}$, and let $p$ be a large prime.
The following was shown in \cite[Section 4.2]{GK13}: $G \simeq D_{4}$,
and for $h_{1}=0, h_{2}=1$, we have $G^{2} = \gal(L^{2}/\fp(T))$
(where $L^2$ denotes the compositum $L_{h_1} L_{h_2}$), and $G^{2}$
genuinely depends 
on $p$.  Namely, for $p \equiv 3 \mod 4$ we have
$G^{2} \simeq D_{4} \times D_{4}$, whereas for $p \equiv 1 \mod 4$,
$G^{2}=G^{2}_{\text{geom}} = H$ is a certain index two subgroup of
$D_{4} \times D_{4}$.  More precisely,
$$
\ifsage
H = \sage{G}
\else
H = \langle (6,7), (2,3)(5,6)(7,8), (1,2)(3,4) \rangle
\fi
\subset D_{4} \times D_{4},
$$
where we have identified the first copy of $D_{4}$ as permutation of
$\{1,2,3,4\}$, and the second copy as a permutation of
$\{ 5, 6 ,7, 8\}$.
As 
\begin{multline*}
D_{4} = \{ 
(1,4)(2,3), (1,3)(2,4), (1,3), (2,4),
(1,2)(3,4), 
{(1,2,3,4)},
{(1,4,3,2)}
 \}
\end{multline*}
(note that $D_{4}$ contains exactly two $4$-cycles) the Chebotarev
density theorem gives that
$$
c(1_{\text{Prime}}) = 2/|D_{4}| = 2/8 = 1/4
$$

A tedious but straightforward calculation gives that $|H|= 32$ and
that
%(here
%we have colored in red all Galois elements corresponding to both
%$f(x)+c$ and $f(x)+1+c$ being irreducible)
there are exactly four elements in $H$  corresponding to
both
$f(x)+t$ and $f(x)+1+t$ being prime for $t \in \fp$, 
namely
${ (1,3,4,2)(5,7,8,6)}$, 
${ (1,3,4,2)(5,6,8,7)}$,
${ (1,2,4,3)(5,6,8,7)}$, and 
${ (1,2,4,3)(5,7,8,6)}$.
%\begin{multline*}
%  \{
%H = \{ \sage{ [ g for g in G ] } \}
% H = 
% \{ (6,7), (1,2)(3,4), (2,3)(5,6)(7,8), (2,3)(5,6,8,7), (2,3)(5,7,8,6),
% \\
% (1,3,4,2)(5,6)(7,8), (1,2,4,3)(5,6)(7,8), (1,2)(3,4)(6,7),
% \\
%{\color{red} (1,3,4,2)(5,7,8,6)}
% , (2,3)(5,7)(6,8), 
%{\color{red} (1,3,4,2)(5,6,8,7)},
%\\
%{\color{red} (1,2,4,3)(5,6,8,7)}, 
%{\color{red} (1,2,4,3)(5,7,8,6)}
% , (5,8), (1,3)(2,4),
% \\
% (1,4)(5,6)(7,8), (1,4)(5,6,8,7), (1,4)(2,3), (1,2,4,3)(5,7)(6,8),
% \\
% (1,3)(2,4)(6,7), (5,8)(6,7), (1,3)(2,4)(5,8), (1,4)(5,7,8,6),
% \\
% (1,3,4,2)(5,7)(6,8), (1,2)(3,4)(5,8), (1,2)(3,4)(5,8)(6,7),
% \\
% (1,4)(2,3)(6,7), (1,3)(2,4)(5,8)(6,7), (1,4)(2,3)(5,8),
% \\
% (1,4)(5,7)(6,8), (1,4)(2,3)(5,8)(6,7) \}
%\end{multline*}
Hence the ``twin prime density'' for the shift $h=1$ equals
$4/|H| = 4/32 = 1/8 \neq 1/4^{2}$.
Similarly, the density for the shift $h = -1$ also equals $1/8$.

As mentioned above, for $p \equiv 3 \mod 4$, the compositum $L^{2}$ of
$L_{0}$ and $L_{1}$ was shown to have maximal Galois group, namely
% $G = \gal(f(x)-t, f(x)+1 -t / \fp(t)) \simeq D_{4} \times D_{4}$.
$G^{2} \simeq D_{4} \times D_{4}$; further the field of constants of
the compositum was shown to equal $\fp[i]$ (where $i^{2}=1$).  Thus,
if $\sigma \in G$ is any element such that $\sigma(i)=-i$, we find
that Frobenius takes values in the coset
$\sigma H \subset D_{4} \times D_{4}$.  In particular, as all elements
of $G$ consisting of two $4$-cycles in $D_{4} \times D_{4}$ are
contained in $H$, there are no such elements in the coset $\sigma H$.
Consequently the Chebotarev density for $(g, g+1)$ both being prime is
{\em zero} for $p \equiv 3 \mod 4$ and $g \in I(f)$ (even though
$c(1_{\text{Prime}})=1/4$.)

On the other hand, the critical points of $f$ (i.e., zeros of $f'$)
are $\{0, -1, 1 \}$, and thus the critical values of $f$ are given by
$R_{f} = \{ 0, -1 \}$.  Hence 
$B(f) = (R_{f}-R_{f}) \setminus \{0 \} = \{-1, 1\}$, and 
thus Theorem~\ref{thm:non-generic-correlations} gives
independence in the sense that the simultaneous prime density for
$g, g+h$ equals $1/4^{2}$ for $h \neq 0, \pm 1$ and $g \in I(f)$.

The ``coincidence'' of getting the expected twin prime density
when averaging over all primes $p$ can be explained as follows.  We
lift the setup to $\Q$ and consider
$G = \gal(f(x)+t, f(x)+1 +t / \Q(t))$.  Then
$G \simeq D_{4} \times D_{4}$, and the constant field extension is
$\Q(i)$.  Thus, if we first average over primes $p$, and then over
$t \in \fp$, the Frobenius element equidistributes in all of $G$ (for
$p \equiv 1 \mod 4$ it equidistributes in $H$, and for
$p \equiv 3 \mod 4$ it equidistributes in the nontrivial coset of $H$,
and $G$ is the union of these two cosets.) In
particular, as there are $4$ elements in $G$ whose cycle structure
corresponds two simultaneous prime specialization, we find that the
$p$-averaged twin prime density equals $4/|G|=4/64=1/4^{2}$, ``as
expected''.

% We've been very brief here: we should add/recall that the field of constants
% (over $\Q$) in this case is $\Q(i)$. Hence for $p \equiv 1 \mod 4$, we
% find that the field of constants over $\fp$ is just $\fp$, and we're
% in the completely geometric situation.  By effective Chebotarev, we
% then find that the proportion of $c \in \fp$ with the desired property
% is given by the proportion of elements in the geometric Galois group 
% $$
% \gal( f(x)-t, f(x)+h-t / \fp(t) ) 
% $$
% acting as a four cycle times a four cycle.

\section{The large $q$ limit}
\label{sec:large-q}

The previous results can be extended to the setting of very short
intervals in $M_{d}(\fq)$ for $q=p^{l}$ as long as $p$ grows
(the key point is that the proof of Lemma~16 in \cite{GK13} also works
for $\fq$ provided $p$ is sufficiently large).

The setting of $p$ fixed and letting $l$ grow is more complicated.  We
first note there is an obvious obstruction to $f(x) + sx$ being Morse
for {\em any} value of $s \in \fq$ in case $p | \deg(f)$ --- clearly
$\deg(f') < d-1$ and hence there are at most $d-2$ critical values.
However, even if we assume $(\deg(f),p) = 1$ there are other
obstructions for the the Galois group being maximal (i.e., that
$G^{k}_{\text{geom}} = S_{d}^{k}$), {\em even though}
$G_{h_{i}} = S_{d}$ for $1 \le i \le k$.
For example, consider the family $f_{s}(x) = x^{3} + sx$ for
$s \in \fq$ where $q = p^{l}$ and $p>3$ is fixed. For all but $O(1)$
choices of $s$, $f_{s}(x)$ is Morse, and $f_{s}'(x)= 3x^{2}+s$ is a
quadratic with two distinct roots in $\fqtwo$, and it is easy to see
that $R_{f_{s}} = \{\alpha_{s}, -\alpha_{s}\}$ for some
$\alpha_{s} \in \fqtwo$.  Taking $k=p$ and letting
$h_{i} = i \alpha_{s}$ for $1 \le i \le k$, we find that the
multiset-union of $R+h_{1}, R+h_{2}, \ldots, R+h_{k}$, as a set
is a linear
$\fp$-subspace in $\fqtwo$, with each element having multiplicity two
(since $|R|=2$).  Consequently $G^{k}_{\text{geom}}$ contains only
{\em even} permutations.  In particular, the equivalence
between cancellation in Möbius sums and Chowla sums
(cf. Theorem~\ref{thm:moebius-chowla-no-cancellation}) does not hold
in the large $q$ limit.
%
% A similar
% construction works also in case $R \subset \fqtwo \setminus \fq$: take
% $k=q^{2}$ and let $h_{1}, \hdots, h_{k}$ be an enumeration of the
% elements in $\fqtwo$.  The multiset union of
% $R+h_{1}, R+h_{2}, \ldots, R+h_{k}$ then equals $\fqtwo$, with each
% element having multiplicity two, and again $G^{k}_{\text{geom}}$
% contains only even permutations.

A more subtle example of independence breaking down can also be given.
For $f(x) = x^4 + x^3 + 3x^2 \in M_{4}({\mathbb F}_{7})$, the critical
values are given by $R = \{0,1,3\}$; taking
%$(h_{1},\ldots, h_{4}) = (0,3,5,6)$ we find that the multiset union
%of
$(h_{1},\ldots, h_{4}) = (0,1,2,4)$ we find that the multiset union of
$h_{i}+R$ has multiplicity two on its support.  This type of example
cannot occur for $p$ large, but if we fix $p$ and consider polynomials
of the form $f_{s}(x) = f(x) + sx$, $s \in \mathbb{F}_{7^{l}}$ for
growing $l$, it is clear that the above phenomena occur at least once
(for $s=0$.)  However, if we {\em fix} $h_{1},\ldots, h_{k}$, in some
extension of $\fp$, this can only happen for $O_{d,k}(1)$ values
$s \in \overline{\fq}$ (but note that this set of ``exceptional''
$s$-values depends on the shifts $h_{1},\ldots, h_{k}$.)

\begin{thm}
  \label{thm:large-q-disjoint}
  Fix distinct elements $h_{1}, \ldots, h_{k} \in \overline{\fp}$, and
  let $f_{0} \in M_{d}(\overline{\fp})$ with $(p,d(d-1))=1$, and let
  $q= p^{l}$ for some $l\geq 1$ large enough so that
  $f_{0} \in \fq[x]$ and $h_{1}, \ldots, h_{k} \in \fq$.  Given
  $s \in \fq$, let $f_{s}(x) = f_{0}(x) + sx$ and for $i=1,\ldots,k$,
  let $K_{i}/\fq(t)$ denote the field extension generated by
  $f_{s}(x)+h_{i}+t$, let $L_{i}$ denote the Galois closure of
  $K_{i}$, and let $L^{k}$ denote the compositum of
  $L_{1},\ldots,L_{k}$.  Then, for all but $O_{d,k}(1)$ values of
  $s \in{\fq}$, $f_{s}$ is Morse, and we have
  $\gal(L^{k}/\fq(t)) \simeq S_{d}^{k}$.
\end{thm}

Before giving the proof of Theorem~\ref{thm:large-q-disjoint} we
deduce an immediate corollary, namely  a somewhat weaker ``large $q$''
analogue of 
Theorems~\ref{thm:prime-independence},
\ref{thm:moebius-chowla-cancellation} and
\ref{thm:class-function-correlations}. %valid for ``generic'' $f_{s}$.
To do so we need some additional notation: given $f\in M_{d}(\fq)$, let
$I_{\fq}(f) := \{ f(x) + a : a \in \fq \}$, and as usual, for $s \in
{\fq}$ let $f_{s}(x) = f(x) + sx$.
\begin{cor}
Fix distinct
  elements $h_{1}, \ldots, h_{k} \in \overline{\fp}$, and let $f_{0}
  \in M_{d}(\overline{\fp})$ with $(p,d(d-1))=1$.
  % and let 
  %  $q= p^{l}$ for   $l\geq 1$ large enough so that $f_{0} \in
  %\fq[x]$ and $h_{1}, 
  %\ldots, h_{k} \in \fq$.
%
  There exists a subset $S_{\text{bad}} \subset \overline{\fp}$,
  depending on $f_{0}$ and $h_{1}, \ldots, h_{k}$, with the following properties:
  
  \begin{enumerate}
  	\item $|S_{\text{bad}}| = O_{d,k}(1)$.
  	\item Let $q=p^l$ be any prime power such that
          $h_1,\dots,h_k\in\fq$ and $f_0\in\fq[x]$.  Then, for
  $s \in \fq \setminus S_{\text{bad}}$, 
  Theorems~\ref{thm:prime-independence},
  \ref{thm:moebius-chowla-cancellation}, and
  \ref{thm:class-function-correlations}
hold for the very short interval $I_{\fq}(f_{s})$.
  For example, if $s \in \fq \setminus S_{\text{bad}}$, then
$$
%\label{eq:prime-k-tuple-q}
|\{ g \in I_{\fq}(f_{s}) : \text{ $g+h_{1}, \ldots, g+h_{k}$ are irreducible }\}|
 = \frac{q}{d^{k}} + O_{d,k}(\sqrt{q}),
 $$
 and
$$
 \sum_{g \in I_{\fq}(f_{s})} 
\left(  \prod_{i=1}^{k}  \mu(g+h_{i})  \right)
= O_{d,k}(\sqrt{q}).
$$
  \end{enumerate}
\end{cor}

% The next proposition will allow us to shows that for any prime $p>2$,
% a fixed polynomial $f$, and a fixed $k$-tuple
% $h_1,\dots,h_k\in\overline{\fp}$, the condition mentioned in
% Proposition~\ref{prop:linear-disjoint-morse} for linear disjointness
% of the splitting fields of $f_s+h_i$ holds for most values of $s$.

As for the proof of Theorem~\ref{thm:large-q-disjoint}, we first show
that critical values having constant difference is a rare occurrence.
\begin{prop}\label{prop:fs-is-lindisjoint}
  Let $f\in M_{d}(\fq)$
where $q=p^{l}$, and assume that
        $p \nmid d(d-1)$.  With $s$ transcendental over
        $\overline{\fp}$, denote  $f_s(x):=f(x)+sx$, and let
        $\tau_1,\tau_2$ be distinct roots of $f_s'(x)=f'(x)+s=0$. Then
        $f_s(\tau_1)-f_s(\tau_2)\notin\overline{\fp}$. 
\end{prop}
\begin{proof}
	Assume by contradiction that
        $c = f_s(\tau_1)-f_s(\tau_2)  \in\overline{\fp}$,  %Let $\fq$
        %be the finite extension of $\fp$ such that $f\in\fq[x]$, 
        and define
        $\mathbb{E}=\fq(c)$. Now, $f_s(x)=f(x)+sx$ and $f_s'(x)=f'(x)+s$
        are irreducible polynomials over $\mathbb{E}(s)$, so for both
        $i=1,2$, $[\mathbb{E}(\tau_i):\mathbb{E}(s)]=d-1$, and
        $[\mathbb{E}(\tau_i):\mathbb{E}(f_s(\tau_i))]=d$. Since the
        degrees of both extensions are co-prime we find that
        $\mathbb{E}(\tau_i)=\mathbb{E}(s,f_s(\tau_i))$, and thus, since
        we assume that $f_s(\tau_1)=f_s(\tau_2)+c$, we find that
        $\mathbb{E}(\tau_1)=\mathbb{E}(\tau_2)$. This implies that
        there exist $A,B,C,D\in \mathbb{E}$ such that
        $\tau_2=\frac{A\tau_1+B}{C\tau_1+D}$. We claim that $C=0$,
        otherwise (note that $f'(\tau_i)=-s$ for $i=1,2$) 
	\begin{multline}
	f(\tau_1)-f'(\tau_1)\tau_1=f(\tau_1)+s\tau_1=f_s(\tau_1)=f_s(\tau_2)+c=f(\tau_2)+s\tau_2+c=\\
	f(\tau_2)-f'(\tau_2)\tau_2 + c =f(\frac{A\tau_1+B}{C\tau_1+D})-\frac{A\tau_1+B}{C\tau_1+D}f'(\frac{A\tau_1+B}{C\tau_1+D}) +c,
	\end{multline}
% 	and so
% 	\begin{equation}
% 	\sum_{m=0}^na_m(1-m)\tau_1^m=\frac{\sum_{m=0}^{n}a_m(1-m)(A\tau_1+B)^m(C\tau_1+D)^{n-m}}{((C\tau_1+D)^n)}
% 	\end{equation}
%	And we get,
        and after clearing denominators we find that $\tau_1$ is a
        root of a polynomial of degree $2d$ (here we use
        $p \nmid d-1$ so that $\deg(f(x)-xf'(x))=d$), with
        coefficients in $\overline{\fp}$, 
%        algebraic over $\mathbb{E}$ (satisfying a polynomial whose
%        leading coefficient is $(1-n)C$,
        contradicting that $\tau_1$ is transcendental.  Therefore
        $C=0$, and thus $\tau_2=A\tau_1+B$ for some
        $A,B\in \mathbb{E}$. Denote $h(x)=f(x)-xf'(x)$. Then
        $h(x)-c=h(Ax+B)$.  Let $R_h$ be the multiset of critical
        values of $h$, and $A_h$ the set of critical points of
        $h$. For any $a\in A_h$, $(a-B)/A$ is a critical point of
        $h(Ax+B)$, and $h(a)$ is a critical value of
        $h(Ax+B)$. Therefore $R_h$ is the multiset of critical values
        also for $h(Ax+B)$. On the other hand, by the equality
        $h(x)-c=h(Ax+B)$, we find that $R_h=R_h-c$.  By
        \cite{JarRaz00} (cf. Claim D' in the proof of
        Proposition~4.3), critical values are distinct and hence $c
        \neq 0$.
        We thus find that there exists a nontrivial $\fp$-action on the multiset
        $R_h$, and therefore $p$ divides the multiset cardinality of
        $R_{h}$, i.e., $p$ divides $\deg(h')=d-1$, contradicting the
        assumption that $p \nmid d-1$.
      \end{proof}
%       \begin{rem}
%         We note that $f(x)=x^{p}$ satisfies $f(x+1) = f(x)+1$
%       \end{rem}
      
      \begin{cor}
        \label{cor:no-ramification-large-q}
  % For any finite field $\fq$ with $q$ odd, any polynomial
  For 
  $f \in M_{d}(\fq)$ such that $(q,d(d-1))=1$, and any set of $k$
  distinct elements 
  $H=\{h_1,\dots,h_k \}\subset\fq$, the set
  $B(f_s)\cap (H-H)$ is empty for all but $O_{d,k}(1)$ values of $s$,
  where $B(f_{s}) = (R_{f_{s}}-R_{f_{s}}) \setminus \{ 0 \}$.
\end{cor}
\begin{proof}
  By Proposition \ref{prop:fs-is-lindisjoint}, for $s$ transcendental
  over $\fp$, $h_i\neq h_j$, and
  $\tau_i\neq\tau_j$ denoting any two distinct roots of $f_s'(x)$,
	$$f_s(\tau_i)-f_s(\tau_j)-(h_i-h_j)\neq0
	$$ 
	Let
	\[
	\Pi(s):=\prod_{h_i\neq h_j}\prod_{\tau_i\neq\tau_j}(f_s(\tau_i)-f_s(\tau_j)-(h_i-h_j))
	\]
	Then $\Pi(s)\neq0$, and as $\Pi(s)$ is a symmetric polynomial in
        the roots of $f'_{s}(x)=0$,  it is a polynomial
        in $s$, of degree bounded in terms of $d$ and $k$. Since
$
B(f_{s})\cap (H-H)\neq\emptyset
$
is equivalent to 
$\Pi(s)=0$, the result follows.  
\end{proof}

Theorem~\ref{thm:large-q-disjoint} now follows easily as the
extensions $L_{1}, \ldots, L_{k}$ are linearly disjoint by
Corollary~\ref{cor:no-ramification-large-q}.

\bibliographystyle{abbrv}
%\bibliography{bibshortint,mypapers}
%\end{document}

\end{document}